\documentclass[12pt]{elsarticle}
\usepackage{amsmath,amsfonts,amssymb}
\usepackage{latexsym}
\usepackage{graphicx,overpic}
\usepackage{subfig}
\usepackage{pgfplots}
\pgfplotsset{compat=newest}
\usepackage{color}
\usepackage{booktabs}
\usepackage{float}
\usepackage{array} 
 
\newcolumntype{x}[1]{>{\hfil$\displaystyle} p{#1} <{$\hfil}} 
\usepackage{array} 
\usepackage{booktabs,amsmath}  

\renewcommand{\div}{\textrm{div}\ \!}

\DeclareGraphicsExtensions{.pdf,.eps,.ps,.eps.gz,.ps.gz,.eps.Y}

\newtheorem{theorem}{Theorem}

\newtheorem{remark}{Remark}[section]
\newtheorem{proposition}{Proposition}[section]
\newenvironment{proof}{\paragraph{Proof}}{\hfill$\square$}

    \makeatletter
\def\ps@pprintTitle{%
	\let\@oddhead\@empty
	\let\@evenhead\@empty
	\let\@oddfoot\@empty
	\let\@evenfoot\@oddfoot
}
\makeatother

\usepackage[normalem]{ulem}

\begin{document}
\begin{frontmatter}
		\title{On stability of Euler flows on \\closed surfaces of positive genus}
	\author{Vladimir Yushutin \\yushutin@math.uh.edu}
	\address{Department of Mathematics, University of Houston, Houston, Texas 77204 }
	\begin{abstract}
	Incompressible flows of an ideal two-dimensional fluid on a closed orientable surface of positive genus are considered. Linear stability of harmonic, i.e. irrotational and incompressible, solutions to the Euler equations is shown using the Hodge-Helmholtz decomposition.  We also demonstrate that any surface Euler flow is stable with respect to harmonic velocity perturbations.
	\end{abstract}

\end{frontmatter}

\section{Introduction}
	Recently{,} there has been a growing interest in the modeling of thin fluid flows {on} curved surfaces. Apart from its practical relevance, mathematical modeling of surface flows reveals many intriguing theoretical issues. 
	In this paper we investigate how the absence of {a} boundary and the surface genus affect the stability of solutions to the Euler equations on a two-dimensional manifold.
	
	Surfaces of positive genus have {``}handles" and admit divergence-free and curl-free (i.e. harmonic) vector fields. Harmonic velocity fields generate incompressible and irrotational flows, which are essential { in} theoretical fluid mechanics. Since the boundary is empty, all steady harmonic vector fields are solutions to Euler equation which {may} potentially drive the surface Euler flows to be unstable. It is worth noting that a harmonic vector field on a closed surface cannot be represented either by the gradient of a potential function or by the skew-gradient of a stream function.

	\section{Hodge-Helmholtz formulation of Euler equation}
	\label{general}
	
	Let ($\mathcal{M}^2$, ${g}$) be a compact connected orientable two-dimensional Riemannian manifold without boundary. {Then t}here exist \cite{ebin1970groups} a smooth \textit{velocity} vector field $v$ and a smooth \textit{static pressure} function $p_s$ on $\mathcal{M}$ such that, for sufficiently smooth vector fields $F$ and $v_0$  on $\mathcal{M}$, the following Euler equations are satisfied:
	\begin{equation}
	\label{Euler_equation}
	\begin{gathered}
	v_t+{\overline{\nabla{}}_{v}}v + \overline{\nabla{}}p_{s} = F\,,\quad
	\div v = 0\,,\quad
	v\big|_{t=0}=v_0\,,
	\end{gathered}
	\end{equation}
	where $\overline{\nabla{}}$ is {the} Levi-Civita connection.
Having chosen {a} volume form $\mu$ associated with $g$, one {may define} {the} Hodge $\star{}$ operator, the Hodge inner product $\left<.,.\right>$, and codifferential $\delta$.
	
	\begin{proposition} Let $f$ be a function, $\alpha, \beta$ be covector fields. Consider $\eta$, a skew-symmetric bilinear form defined on $\mathcal{M}^2$ equipped with volume form $\mu$. {Then the} Hodge 	$\star$ operator  maps vectors to vectors, functions to bilinear forms and vice versa, and the following properties hold true \cite{abraham2012manifolds}:
\begin{equation*}
				\begin{aligned}[c]
					\star\!\star\!f&=f\\
				\star\!\star\!\alpha &= -\alpha\\
					\delta{}&=-\star\!d\star\\
					\left<\alpha,\beta\right >&=\int_\mathcal{M}	(\alpha, \beta) \wedge \mu
				\end{aligned}
				\qquad
				\qquad
				\qquad
				\begin{aligned}[c]
						\star \mu&=1\\
				\alpha \wedge \star \beta{}&=(\alpha, \beta) \wedge \mu\\
			\left<\delta\alpha,\beta\right >&=\left<\alpha,d\beta\right >\\
			\left<f,\star{}\eta\right >&=\int_\mathcal{M}	f \wedge \eta
				\end{aligned}
\end{equation*}
\end{proposition}
	
	The \textit{rotational} form of the dual Euler equations \eqref{Euler_equation} on a two-dimensional manifold can be obtained in terms of the {covector field $v^{\flat}$ dual to $v$} as follows:
	\begin{equation}
\label{dual_Euler_bernoulli}
\begin{gathered}
v^\flat_t + \star{}dv^\flat{}\wedge{}\star{}\!v^\flat{} +dp= F^\flat{}\,,\qquad
-\delta{}v^\flat=0\,,\qquad
v^\flat{}\big|_{t=0}=v_0^\flat\,,
\end{gathered}
\end{equation}
where vorticity function $\omega=\star{}dv^\flat{}$ is the curl of vector field $v$, and $p=p_s+\frac{1}{2} v^\flat{}(v )$ is {the} Bernoulli pressure.

 Hodge theory implies that an incompressible covector field on $\mathcal{M}$ can be uniquely decomposed into $L^2$-orthogonal terms:
		\begin{equation}
	\label{helm-hodge}
	v^\flat =\star{}d\psi  + \gamma{}\,,
\end{equation}
where {$\psi$} is {the} stream function and $\gamma$ is a harmonic covector field. By substituting \eqref{helm-hodge} into \eqref{dual_Euler_bernoulli},
one derives the Hodge-Helmholtz formulation of the Euler equations (initial condition{s} omitted):
\begin{equation}
\begin{aligned}
\label{dual_streamfunction}
\star{}d\psi{}_t+\gamma{}_t + \omega{}\wedge{}\left(-d\psi{}  + \star\gamma{} \right) +dp = F^\flat{}\,,\qquad
-\delta{}d\psi{} =\omega\,.
\end{aligned}
\end{equation}

\begin{remark}
The Hodge-Helmholtz decomposition of the first equation of \eqref{dual_streamfunction} produces three independent relations thus matching with the number of unknowns ($\psi$,~$\omega{}$,~$\gamma{}$,~$p$). \end{remark}


\section{Stability of harmonic solutions}\label{stability_harmonic}
Harmonic covector fields, {i.e} $v^\flat=\gamma{}$, have zero vorticity and thus represent all incompressible and irrotational velocity fields. The space of all harmonic vector {fields} is finite-dimensional, and its dimension {is equal to} the second Betti number{,} which is the doubled number of {``}handles" for a closed oriented surface with only one connected component; therefore, only a positive genus surface admits harmonic vector fields. 
Note that the image of a harmonic vector under {the} Hodge star operator is {again} harmonic.

We investigate stability of any harmonic solution ($\psi$, $\omega{}$, $\gamma{}$, $p$)=($0$, $0$, $\gamma{}_0$, $p_0$) of the Euler equation \eqref{dual_streamfunction} by the standard linearization:
\begin{align}
\label{harm_steady_sol}
\psi=0+\varepsilon\widetilde{\psi}+O(\varepsilon^2)\,,\quad
\gamma=\gamma_0+\varepsilon\widetilde{\gamma}+O(\varepsilon^2)\,,\quad p=p_0+\varepsilon\widetilde{p}+O(\varepsilon^2)\,.
\end{align}

\begin{theorem}
	\label{harmonic_stability}
	A harmonic Euler flow on a closed surface is at most polynomially (linearly) unstable in {the} $L^2$ norm.
\end{theorem}
\begin{proof}
		To prove the statement, one estimates both parts of the velocity perturbation $\widetilde{v}^\flat=\star{}d\widetilde{\psi}+\widetilde{\gamma}$. {To begin o}ne substitutes \eqref{harm_steady_sol} into \eqref{dual_streamfunction} and derives the system that governs the evolution of small perturbations:
\begin{equation}
\begin{split}
\label{stab_harm}
\star{}d\widetilde{\psi}_t+\widetilde{\gamma}_t+ \widetilde{\omega}\wedge{}\star{}\!{\gamma_0}+d\widetilde{p}=0\,,\qquad
-\delta{}d\widetilde{\psi{}} =\widetilde{\omega}
\end{split}
\end{equation}

We first estimate the stream function perturbation by taking  $\star{}d$ of the first {equation in} \eqref{stab_harm} and obtain the transport equation for the vorticity perturbation along the main flow ${\gamma_0}$:
\begin{equation}
\label{transport}\widetilde{\omega}_t + \star{}(d\widetilde{\omega}\wedge\star{}{\gamma_0})=0
\end{equation}
Note that $L^2$-norm of $\widetilde{\omega}$ remains constant over time:
\begin{equation}
\begin{split}
-\left<\widetilde{\omega},\widetilde{\omega}\right>_t=-2\left<\widetilde{\omega}_t,\widetilde{\omega}\right>=2\left<\star{}(d\widetilde{\omega}\wedge\star{}{\gamma_0}),\widetilde{\omega} \right>=2\int_\mathcal{M} \widetilde{\omega}\wedge{}d\widetilde{\omega}\wedge\star{}{\gamma_0}&=\\=\int_\mathcal{M} d(\widetilde{\omega}\wedge{}\widetilde{\omega})\wedge\star{}{\gamma_0}=\left<d(\widetilde{\omega}\wedge{}\widetilde{\omega}),\gamma_0\right>=\left<\widetilde{\omega}\wedge{}\widetilde{\omega}, \delta{}\gamma_0\right>&=0
\end{split}
\label{vort_pert}
\end{equation}

 Using the first smallest positive eigenvalue $\lambda_{min}$ of Hodge-Laplacian $\delta{}d$, one bounds {the} $L^2$-norm of $\widetilde{\psi}$ by {the} $L^2$-norm of $\widetilde{\omega}$:
	\begin{align*}
	{\lambda_{min}}\left<\widetilde{\psi},\widetilde{\psi}\right>\leq\left<\widetilde{\psi},\delta{}d\widetilde{\psi}\right>&=\left<\widetilde{\psi},-\widetilde{\omega}\right>	\leq{} \left<\widetilde{\psi},\widetilde{\psi}\right>^\frac12 \left<\widetilde{\omega},\widetilde{\omega}\right>^\frac12,\\
	\left<\widetilde{\psi},\widetilde{\psi}\right>&\leq\frac{1}{\lambda_{min}^2}\left<\widetilde{\omega},\widetilde{\omega}\right>
			\end{align*}

	The evolution of the harmonic part of perturbation $\widetilde{\gamma}$ is studied by means of a Hodge-orthonormal basis $\{h^i\}$ of the finite-dimensional space of harmonic covector fields (summation over $i$ is assumed): 
	\begin{equation}\label{basis}
\widetilde{\gamma}=c_i(t) h^i\,.
	\end{equation}

	Taking {the} Hodge inner product of \eqref{stab_harm} with a basis covector field $h^k$ results in the following equation for $c_k$:
	$$\dot{c}_k(t)+\int_\mathcal{M}\widetilde{\omega}\wedge{}(\star{}\gamma{}_0,h^k)\mu=0$$
Note that $(\star{}\gamma{}_0,h^k)$ are constant functions. We {now} show that $\widetilde{\gamma}$ grows at most linearly in time using Cauchy-Schwarz inequality:
		$$|\dot{c}_k(t)|\leq\left(\int_\mathcal{M}\widetilde{\omega}^2\mu\right)^{1/2} \left(\int_\mathcal{M} (\star{}\gamma{}_0,h^k)^2\mu\right)^{1/2}=\left<\widetilde{\omega},\widetilde{\omega}\right>^\frac12\left(\int_\mathcal{M} (\star{}\gamma{}_0,h^k)^2\mu\right)^{1/2}\,,$$ 
	and {since} the right-hand side is constant over time due to \eqref{vort_pert}{, this concludes the proof}.
\end{proof}

\section{Stability of Euler flows with respect to harmonic perturbations}
\label{stability_harmonic_perturb}
Theorem \ref{harmonic_stability} shows that harmonic solutions to the Euler equations on a closed surface are linearly unstable. In this section{,} we would like to address the following question: is it possible {for} harmonic perturbations {to} destabilize an Euler flow on a closed surface with positive genus?

As in Section \ref{stability_harmonic}, let us perturb a solution ($\psi_0$, $\omega{}_0$, $\gamma{}_0$, $p_0$) of  \eqref{dual_streamfunction} by a harmonic perturbation $\widetilde{v}^\flat=\widetilde{\gamma}$:
\begin{equation}
\label{harm_perturbed}
\widetilde{\gamma}_t + \omega_0\wedge{}\star{}\!\widetilde{\gamma}+d\widetilde{p}=0\,.
\end{equation}

\begin{theorem}
	Any Euler flow on a closed surface is $L^2$-stable with respect to harmonic perturbations.
\end{theorem}
\begin{proof} We employ decomposition \eqref{basis}.
Taking {the} Hodge inner product of \eqref{harm_perturbed} with a basis vector $h^i$ results in the following finite-dimensional system:
\begin{equation}
\label{fin_system}
\dot{c}_i=c_j\int_\mathcal{M}\omega_0\wedge{}h^i\wedge{}h^j\,.
\end{equation}
The matrix $A_{ij}=\int_\mathcal{M}\omega_0\wedge{}h^i\wedge{}h^j$ is skew-symmetric and even-dimensional. Hence the finite-dimensional system \eqref{fin_system} preserves the norm of the vector $c\,$: $\left|c\right|^2_t=c\cdot{}c_t=c\cdot Ac=A^Tc\cdot{}c=0$, and the harmonic perturbation $\widetilde{\gamma}$ remains bounded.
\end{proof} 

\section{Conclusions}
We have investigated the influence of a surface genus on the stability of solutions to surface Euler equations.  It was shown that harmonic solutions are at most linearly unstable with a factor proportional to the $L^2$-norm of {a} {perturbation} vorticity. We also prove{d} that Euler flows on a closed surface are not destabilized by harmonic perturbations.
\section{Acknowledgments}
The author gratefully thanks Boris Khesin who provided expertise that greatly assisted the research and also Maxim Olshanskii for constant interest in the work and helpful suggestions. 

\bibliographystyle{elsarticle-num}
\bibliography{bibliography}{}	

\begin{thebibliography}{1}
\expandafter\ifx\csname url\endcsname\relax
  \def\url#1{\texttt{#1}}\fi
\expandafter\ifx\csname urlprefix\endcsname\relax\def\urlprefix{URL }\fi
\expandafter\ifx\csname href\endcsname\relax
  \def\href#1#2{#2} \def\path#1{#1}\fi

\bibitem{ebin1970groups}
D.~G. Ebin, J.~Marsden, Groups of diffeomorphisms and the motion of an
  incompressible fluid, Annals of Mathematics (1970) 102--163.

\bibitem{abraham2012manifolds}
R.~Abraham, J.~E. Marsden, T.~Ratiu, Manifolds, tensor analysis, and
  applications, Vol.~75, Springer, 2012.

\end{thebibliography}
\end{document}